\documentclass[12pt,a4paper]{amsart}

\usepackage{amssymb}
\usepackage{amsfonts}
\usepackage{amsmath}
\usepackage[utf8]{inputenc}
\usepackage[mathscr]{euscript}
\usepackage{enumerate}
\usepackage{tikz-cd}

\usepackage[margin=1in,headheight=13.6pt]{geometry}

\numberwithin{equation}{section}

\newtheorem{theorem}[equation]{Theorem}
\newtheorem{lemma}[equation]{Lemma}
\newtheorem{proposition}[equation]{Proposition}
\newtheorem{corollary}[equation]{Corollary}

\theoremstyle{definition}
\newtheorem{definition}[equation]{Definition}

\theoremstyle{remark}
\newtheorem{remark}[equation]{Remark}

\newcommand{\uset}[1]{\text{$\uparrow\hspace{-0.1cm}{#1}$}}				
\newcommand{\usetr}[2]{\text{$\uparrow_{\scriptscriptstyle{#2}}\hspace{-0.1cm}{#1}$}} 
\newcommand{\lset}[1]{\text{$\downarrow\hspace{-0.1cm}{#1}$}}				
\newcommand{\lsetr}[2]{\text{$\downarrow_{\scriptscriptstyle{#2}}\hspace{-0.1cm}{#1}$}} 
\newcommand{\scj}{\subseteq}

\newcommand{\til}[1]{\widetilde{#1}}

\newcommand{\cov}{\mathcal{C}}
\newcommand{\covd}{\mathcal{D}}
\newcommand{\covt}{\mathcal{T}}
\newcommand{\germ}{\text{germ}}
\newcommand{\gr}[1]{\mathcal{G}(#1)}
\newcommand{\dom}[1]{\mathbf{d}(#1)}
\newcommand{\ran}[1]{\mathbf{r}(#1)}
\newcommand{\Sp}[1]{\mathsf{Sp}(#1)}
\newcommand{\grp}{\mathcal{G}}
\newcommand{\bis}{\mathsf{B}}
\newcommand{\filt}{\mathsf{F}}
\newcommand{\tfilt}{\mathsf{T}}
\newcommand{\ufilt}{\mathsf{U}}
\newcommand{\grl}{\mathcal{L}}
\newcommand{\grh}{\mathcal{H}}

\newcommand{\Image}{\operatorname{Im}}

\begin{document}
	
\title{Coverages on Inverse Semigroups}
\author{Gilles G. de Castro}
\address{Departamento de Matemática, Centro de Ciências Físicas e Matemáticas, Universidade Federal de Santa Catarina, 88040-970 Florianópolis SC, Brazil.}
\email{gilles.castro@ufsc.br}


\begin{abstract}
First we give a definition of a coverage on a inverse semigroup that is weaker than the one gave by a Lawson and Lenz and that generalizes the definition of a coverage on a semilattice given by Johnstone. Given such a coverage, we prove that there exists a pseudogroup that is universal in the sense that it transforms cover-to-join idempotent-pure maps into idempotent-pure pseudogroup homomorphisms. Then, we show how to go from a nucleus on a pseudogroup to a topological groupoid embedding of the corresponding groupoids. Finally, we apply the results found to study Exel's notions of tight filters and tight groupoids.
\end{abstract}

\maketitle

\section{Introduction}

Stone's duality between Boolean algebras and what is now called Stone spaces \cite{MR1507106} was one of the first steps in studying the topology on a set in a more algebraical setting. His work was vastly generalized, for example to the duality between sober spaces and spatial frames \cite{MR698074,MR2868166,MR1002193}. In this context, Johnstone's definition of a coverage on a semilattice \cite{MR698074}, used to present frames using generators and relations, can be also used to impose relations on open sets of a topological space.

Another result connecting topological spaces with algebras is the Gelfand-Naimark duality between commutative C*-algebras and locally compact Hausdorff spaces \cite{MR0009426} by means of a subalgebra of the algebra of all complex valued continuous functions from a topological space. One can then think that a noncommutative C*-algebra represents a complex valued functions of a virtual object thought to be a noncommutative space. This led Connes to the development of Noncommutative Geometry \cite{MR1303779}. A concrete mathematical object that can play the role of the noncommutative space is a topological groupoid \cite{MR1303779,MR584266}, and it is often useful to describe a C*-algebra using groupoids in order to use the toolkit started in \cite{MR584266}.

One interest problem that arises is to generalize the duality between topological spaces and frames by replacing topological spaces with topological groupoids (or even more general objects). One major milestone in this direction is the work of Resende \cite{MR2304314}, where he relates étale groupoids, quantales and (abstract) pseudogroups. In this paper we are more interested in the relation between étale groupoids and pseudogroups, the latter being a certain subclass of inverse semigroups.

Going back to C*-algebras, inverse semigroups already appears in Renault's monograph \cite{MR584266} and it is one tool in describing a C*-algebras as in Paterson's book \cite{MR1724106}. Paterson's universal groupoid from an inverse semigroup often did not give the correct groupoid to describe several classes of C*-algebras, but is was needed to consider a reduction of this groupoid such as he has done in \cite{MR1962477} for graphs. This reduction was described in terms of the object used to build an inverse semigroup, and not in terms of the inverse semigroup itself. This led Exel to define the notion of tight filters and tight groupoids \cite{MR2419901}. The tight groupoid is then the correct groupoid in several cases such as for higher-rank graphs \cite{MR2419901} and for labelled spaces \cite{MR3648984,MR3680957,Gil3} both generalizing the graph case. Another possible generalization for graphs is that of topological graphs defined by Katsura in \cite{MR2067120}, however independently of the choice of the inverse semigroup, the tight groupoid will not coincide, in general, with Yeend's groupoid from a topological graph \cite{MR2301938}, since the former is always an ample groupoid, whereas the latter needs not to be. This was the author's main motivation to \cite{de_castro_2020}, where the unit space of Yeend's groupoid was described using Johnstone's notion of coverage on a semilattice.

In \cite[Section 2]{de_castro_2020}, the author showed how to use coverages on semilattices in order to impose relations on the topology of a topological space. One of the main goal of this paper is to generalize this construction by replacing semilattice with inverse semigroup, topological space with étale groupoid and topology with the set of open bisections on the groupoid. For that we need a notion of coverage on a inverse semigroup. This was done in \cite{MR3077869} by Lawson and Lenz, however their definition is not exactly a direct generalization of coverage on a semilattice since they impose some more restrict conditions. Also they were able to show the existence of a pseudogroup satisfying a universal property of transforming cover-to-join maps to pseudogroup maps by imposing an extra condition on the coverage. In Section \ref{sec:coverage}, we give a new definition of coverage on a inverse semigroup that is a direct generalization of Johnstone's definition. For the existence of the pseudogroup satisfying a universal property, instead of imposing more conditions of the coverage, we impose more conditions on the maps, namely, we ask the maps to be idempotent-pure. This will not hinder the results of the remainder of the paper.

Connected to of coverages, both in the semilattice and in the inverse semigroup cases, is the notion of a nucleus. In order to impose relations on the topology of a topological space, we define a nucleus on the topology. This gives a frame homomorphism, and we can use the spectrum functor that gives the duality between frames and topological spaces. Resende's work in building a pseudogroup from an étale groupoid and vice-versa in \cite{MR2304314} was the level of objects only. One approach of adding morphisms was done by Lawson and Lenz in \cite{MR3077869}, where they defined callitic maps between pseudogroups, however they do not discuss if a nucleus on a pseudogroup is a callitic map. So instead of relying on the functor built in \cite{MR3077869}, in Section \ref{sec:nucleus} we show directly how to go from a nucleus on a pseudogroup to a topological groupoid homomorphism of the corresponding étale groupoids. In fact, the proof works essentially the same for a nucleus on a inverse semigroup and the universal groupoid associated to it.

In Section \ref{sec:imposing}, we work on the problem of imposing join relations on the bisections of an étale groupoid. These relations can be restricted to relations on the topology of the unit space where we can apply the results of \cite[Section 2]{de_castro_2020}. In fact we show that the resulting groupoid after imposing the relations can be seen as a reduction of the original groupoid to a subspace of the unit space found from the results of \cite[Section 2]{de_castro_2020}.

Finally, in Section \ref{sec:tight}, we apply some of the theory developed in this paper to study tight filters and tight groupoids as defined by Exel in \cite{MR2419901}.


\section{Preliminaries}

The main purpose of this section is to recall some of the needed definitions as well as establish some notations. The first subsection is based on \cite{MR698074,MR1002193}, the second is based on \cite{MR1694900,MR1724106,resende2006lectures,MR2304314}, and the third is based on \cite{MR698074,MR3109745,MR3077869,MR2465914,MR1724106,MR1002193}.

\subsection{Posets and filters}

For a poset $(P,\leq)$ and $A\scj P$, the \emph{upper set} of $A$ is the set $\uset{A}=\{p\in P\mid \exists a\in A,a\leq p\}$, and the \emph{lower set} of $A$ is the set $\lset{A}=\{p\in P\mid \exists a\in A,p\leq a\}$. If we need to specify the poset, we also use the notations $\usetr{A}{P}$ and $\lsetr{A}{P}$, respectively. We say that $A$ is \emph{upper closed} if $A=\uset{A}$ and \emph{lower closed} if $A=\lset{A}$. A lower closed set will also be called an \emph{order ideal}.\footnote{Usually, for a set $A$ to be an order ideal we also ask for it to be non-empty and directed, that is for every $a,b\in A$, there exists $c\in A$ such that $a\leq c$ and $b\leq c$. For this paper, we use the same terminology as in \cite{MR3077869}.}

For $A\scj P$, we say that $x$ is a \emph{lower bound} of $A$ if $x\leq a$ for all $a\in A$. The greatest lower bound of $A$, if it exists, will be called the \emph{meet} of $A$ and will be denoted by $\bigwedge A$. Similarly, we define a \emph{upper bound} and the \emph{join} of a set $A$, denoted by $\bigvee A$, if it exists, is the least upper bound. We use the notations $a\wedge b$ and $a\vee b$ to mean $\bigwedge \{a,b\}$ and $\bigvee \{a,b\}$, respectively. A lower bound for $P$, if it exists, is unique and called the \emph{minimum} of $P$, usually denoted by $0$. In this case $\bigvee \emptyset=0$. Analogously, a \emph{maximum} for $P$, if it exists, is the unique upper bound for $P$ and usually denoted by $1$. In this case $\bigwedge\emptyset=1$.

We say that $P$ is a \emph{meet-semilattice}, or simply, a \emph{semilattice}, if $a\wedge b$ exists for every $a,b\in P$. If arbitrary meets and joins exist, we say that $P$ is a \emph{complete lattice}. A \emph{frame} $F$ is a complete lattice satisfying the following distributive law: $a\wedge \bigvee B=\bigvee\{a\wedge b\mid b\in B\}$ for every $a\in F$ and $B\scj F$.

A \emph{filter} in a poset $P$ is a non-empty proper subset $\xi\scj P$ such that $\xi$ is upper closed and for every $a,b\in\xi$, there exists $c\in \xi$ such that $c\leq a$ and $c\leq b$. Note that if $P$ is a semilattice, then a non-empty proper upper closed subset $\xi\scj P$ is a filter if and only if $a\wedge b\in\xi$ for every $a,b\in\xi$. A filter $\xi$ on a poset $P$ is called \emph{ultrafilter} if whenever $\eta$ is a filter in $P$ such that $\xi\scj\eta$, we have that $\xi=\eta$. A filter $\xi$ in a complete lattice $P$ is called \emph{completely prime} if for every $A\scj P$ such that $\bigvee A\in\xi$, there exists $a\in A$ such that $a\in\xi$.

\subsection{Inverse semigroups and pseudogroups}\label{subsec:inv.sem}

An \emph{inverse semigroup} $S$ is a semigroup such that for every $x\in S$, there is a unique $y\in S$ such that $xyx=x$ and $yxy=y$. The element $y$ will be denoted by $x^{-1}$. We say that $S$ has a $0$ if there exists, a necessarily unique, element $0\in S$ such that $0x=0=x0$ for every $x\in S$. An inverse semigroup with an identity is called an \emph{inverse monoid}. For an inverse semigroup $S$, the relation $x\leq y$ if $x=yx^{-1}x$ is a partial order on $S$. For other properties and equivalent formulations of the partial order we refer the reader to \cite[Section 1.4]{MR1694900}. Notice that for an inverse semigroup with $0$, for any filter $\xi$ on $S$, we have that $0\notin \xi$, since $\uset{0}=S$. A semigroup homomorphism between inverse semigroups preserves inverse.

An element $e$ on a semigroup $S$ is called an \emph{idempotent} if $ee=e$. The set of all idempotents of $S$ will be denoted by $E(S)$. If $S$ is an inverse semigroup, we have that $E(S)=\{x^{-1}x\mid x\in S\}$, moreover the product restricted to $E(S)$ is commutative, and with the order induced by the order on $S$, we have that $E(S)$ is a semilattice with meet given by the product, that is, $e\wedge f=ef$ for every $e,f\in E(S)$. On the other hand, every semilattice can be thought as being an inverse semigroup consisting only of idempotents by considering the meet of two elements as the operation.  We say that a semigroup homomorphism $\phi:S\to T$ is \emph{idempotent-pure} if for every $s\in S$ and $t\in E(T)$ such that $\phi(s)=t$, we have that $s\in E(S)$.

For an inverse semigroup $S$, $x,y\in S$ are called \emph{compatible} if $x^{-1}y,xy^{-1}\in E(S)$. A subset $A\scj S$ is called \emph{compatible} if for every $a,b\in A$, we have that $a$ and $b$ are compatible. It can be shown that for a non-empty subset $A\scj S$, if $\bigvee A$ exists, then $A$ is compatible. We say that $S$ is \emph{infinitely distributive} if for every non-empty subset $A\scj S$ such that $\bigvee A$ exists and for every $x\in S$, we have that $\bigvee xA$ and $\bigvee Ax$ also exist, and the following equalities hold $x\bigvee A = \bigvee xA$ and $\bigvee (Ax)=\left(\bigvee A\right)x$.

A \emph{pseudogroup} $P$ is an infinitely distributive inverse monoid with $0$ such that for every non-empty compatible subset $A\scj P$ we have that $\bigvee A$ exists. Notice that, in this case, $E(P)$ is a frame. A \emph{pseudogroup homomorphism} is a semigroup homomorphism between pseudogroups that preserves joins of compatible sets. A filter $\xi$ on $P$ will be called \emph{completely prime} if for every non-empty compatible subset $A\scj P$ such that $\bigvee A\in\xi$, there exists $a\in A$ such that $a\in A$.

\subsection{Groupoids of filters}\label{subsec:groupoid}

A \emph{groupoid} is a set $\grp$ together with a partially defined binary operation $\cdot:\grp^{(2)}\scj\grp\times\grp\to\grp$, called multiplication, and unary operation $^{-1}:\grp\to\grp$, called inverse, satisfying the following conditions:
\begin{itemize}
	\item $(g^{-1})^{-1}=g$ for every $g\in \grp$;
	\item if $(f,g),(g,h)\in\grp^{(2)}$, then $(fg,h), (f,gh)\in\grp^{(2)}$ and $(fg)h=f(gh)$;
	\item for every $f\in\grp$, we have that $(f,f^{-1})\in\grp$ and if $(f,g)\in\grp^{(2)}$, then $f(gg^{-1})=f$ and $(f^{-1}f)g=g$.
\end{itemize}
We define the set of \emph{units} of $\grp$ as $\grp^{(0)}=\{gg^{-1}\mid g\in\grp\}$. We define two maps $\mathbf{d}:\grp\to\grp^{(0)}$ and $\mathbf{r}:\grp\to\grp^{(0)}$ by $\dom{g}=g^{-1}g$ and $\ran{g}=gg^{-1}$ respectively called the domain and the range maps. In fact, a groupoid can be seen as a small category for which every morphism is an isomorphism. The set $\grp^{(2)}$ can be described as the set of all pairs $(f,g)$ such that $\dom{f}=\ran{g}$. A \emph{groupoid homomorphism} between groupoid $\grp$ and $\grh$ is a map $\phi:\grp\to\grh$ such that $(f,g)\in\grp^{(2)}$ implies that $(\phi(f),\phi(g))\in\grh^{(2)}$ and $\phi(fg)=\phi(f)\phi(g)$, or from the category point of view, $\phi$ is a functor. A \emph{topological groupoid} is a groupoid $\grp$ with a topology for which the multiplication and the inverse are continuous. If moreover, the maps $\mathbf{d}$ and $\mathbf{r}$ are local homeomorphisms, we say that $\grp$ is an \emph{étale groupoid}.

For an inverse semigroup $S$, let $\grl(S)=\{A\scj S\mid A\text{ is a filter in }S\}\cup\{S\}$. For $A\in\grl(S)$, define $\dom{A}=\uset{A^{-1}A}$ and $\ran{A}=\uset{AA^{-1}}$. Also define $\grl(S)^{(2)}=\{(A,B)\mid \dom{A}=\ran{B}\}$, a partially defined multiplication by $A*B=\uset{AB}$ and an inverse by $A^{-1}=\{a^{-1}\mid a\in A\}$. With this structure $\grl(S)$ is a groupoid. The set of units $\grl(S)$ can be described as the elements $A\in\grl(S)$ such that $A\cap E(S)\neq\emptyset$. We will consider two topologies on $\grl(S)$. For an element $s\in S$, define $U_s=\{A\in\grl(S)\mid s\in A\}$, then the family $\{U_s\}_{s\in S}$ is a basis for a topology with will denote by $\tau_S$. For $s,s_1,\ldots,s_n\in S$ with $s_i\leq s$ for all $i=1,\ldots,n$, we define $U_{s:s_1,\ldots,s_n}=U_s\cap U_{s_1}^c\cap\cdots\cap U_{s_n}^c$. The family of all such sets is a basis for another topology which will be called the \emph{patch topology} and will be denoted by $\tau_{patch}$. In general, $\tau_{patch}$ if finer than $\tau_S$. Paterson's universal groupoid for an inverse semigroup can be identified with $(\grl(S),\tau_{patch})$ and it will be denoted by $\grp_u(S)$.

\begin{remark}
	If $S$ is an inverse semigroup with 0, we can show that if $A,B$ are filters in $S$ such that $\dom{A}=\ran{B}$, then $0\notin AB$, so that $\uset{AB}$ is also a filter. In this case, we define $\grl(S)$ to be only the set of filters, since if $S$ were to be an element of $\grl(S)$ it would be an isolated because $U_0$ would be $\{S\}$. With respect to the universal property of $\grp_u(S)$, this has to do with inverse semigroup homomorphisms preserving zeros or not.
\end{remark}

For a pseudgroup $P$, we define the groupoid $\gr{P}$ as the set of all completely prime filters on $P$ with the same structure defined for $\grl(P)$. For $\gr{P}$, we will only consider the topology given by the sets $V_s=U_s\cap\gr{P}$, where $U_s$ is as above.

For an étale groupoid $\grp$, an open bisection of $\grp$ is an open subset $A\scj\grp$ such that the maps $\mathbf{d}$ and $\mathbf{r}$ are one-to-one when restricted to $A$. The set of all open bisections $\bis(\grp)$ is a pseudogroup with the natural multiplication. For an element $g\in\grp$ the set $F_g=\{A\in\bis(\grp)\mid g\in A\}$ is a completely prime filter. The map $\eta:\grp\to\grp(\bis(\grp))$ given by $\eta(g)=F_g$ is a well-defined groupoid homomorphism. We say that $\grp$ is \emph{sober} if $\eta$ is a homeomorphism.

\begin{remark}
	A topological space $X$ can be seen as a groupoid by considering $X=\grp=\grp^{(0)}$ and $\grp^{(2)}$ the diagonal of $X\times X$. In this case $\bis(X)$ is the topology on $X$ and the above definition of sober coincides with the usual one for topological spaces. An equivalent definition of sober space is a $T-0$ space such that for every meet-irreducible open proper subset $U$ of $X$, there exits $x\in X$ such that $U=X\setminus\overline{\{x\}}$ \cite{MR2868166}. On the other direction, if $F$ is a pseudogroup consisting only of idempotents, then $F$ can be seen as a frame. In this case $\grp(F)$ is a topological space which is called the \emph{spectrum} of $F$ and will be denoted by $\Sp{F}$.
\end{remark}


\section{Coverages on inverse semigroups}\label{sec:coverage}

In this section we give a new definition of coverage on a inverse semigroup and we show that there exists a pseudogroup satisfying a certain universal property with respect to this coverage.

\begin{definition}\label{def:coverage}
	Let $S$ be a inverse semigroup and $\cov=\{\cov(a)\}_{a\in S}$ be a family where $\cov(a)$ is a set of subsets of $\lset{a}$. We say that $\cov$ is a \emph{coverage} on $S$ if for every $a,b\in S$ and $X\in\cov(a)$, we have that $bX\in\cov(ba)$ and $Xb\in\cov(ab)$. In this case, the elements of $\cov(a)$ are called \emph{coverings} of $a$. And we say that $\cov$ is a \emph{strong coverage} if the following conditions are satisfied:
	\begin{itemize}
		\item[(R)] $\{a\}\in\cov(a)$ for all $a\in S$.
		\item[(I)] If $X\in\cov(a)$ then $X^{-1}\in\cov(a^{-1})$.
		\item[(MS)] If $X\in\cov(a)$ and $Y\in\cov(b)$ then $XY\in\cov(ab)$.
		\item[(T)] If $X\in\cov(a)$ and $X_i\in\cov(x_i)$ for each $x_i\in X$ then $\bigcup_i X_i\in \cov(a)$.
	\end{itemize}
\end{definition}

Using axioms (R) and (MS), we see that every strong coverage is a coverage.

Given two coverages $\cov$ and $\covd$ on an inverse semigroup $S$, notice that if we define a family $\cov\cup\covd=\{(\cov\cup\covd)(a)\}_{a\in S}$ by $(\cov\cup\covd)(a)=\cov(a)\cup\covd(a)$, we get a new coverage on $S$.

\begin{remark}
	What we are calling strong coverage is the definition of a coverage given by Lawson and Lenz in \cite{MR3077869}. There are a couple of reasons why we're giving a new definition. The first is that definition of a coverage given by Johnstone in \cite{MR698074} for semilattices does not necessarily satisfies conditions (MS) and (T) so it is not a coverage in Lawson and Lenz's sense, but it is with respect to Definition \ref{def:coverage}. The second reason has to do with the union of coverages $\cov$ and $\covd$ explained above. Again, $\cov\cup\covd$ does not necessarily satisfies conditions (MS) and (T). The union of coverages will be important in Section \ref{sec:imposing}. The price we have to pay in order to work with this new definition is that we have to weaken the universal property of the pseudogroup generated by a coverage.
\end{remark}

The next proposition shows how to relate coverages on an inverse semigroup and coverages on its idempotents semilattice.

\begin{proposition}\label{prop:restriction.cov.idempotents}
	Let $S$ be an inverse semigroup and $E:=E(S)$ its semilattice of idempotents. For a coverage $\cov$ on $S$, let $\cov_{E}$ be the restriction to $E$. Then, the map $\cov\mapsto\cov_{E}$ is a bijection between coverages on $S$ and coverages $\covd$ on $E$ such that for every $e\in E$, $s\in S$ and $X\in\covd(e)$, we have that $sXs^{-1}\in\covd(ses^{-1})$.
\end{proposition}

\begin{proof}
	By definition of coverage, if $\cov$ is a coverage on $S$, then $sXs^{-1}\in\cov_{E}(ses^{-1})$ for every $e\in E$, $s\in S$ and $X\in\cov_{E}(e)$.
	
	Now suppose we are given a coverage $\covd$ on $E$ satisfying the above property. Define a family $\til{\covd}=\{\til{\covd}(s)\}_{s\in S}$ by $\til{\covd}(s)=\{sX\mid X\in\covd(s^{-1}s)\}\cup\{Xs\mid X\in C(s^{-1}s)\}$ and notice that $\til{\covd}(e)=\covd(e)$ for all $e\in E$. Since, for $e\in E$ and $s\in S$, $e\leq s^{-1}s$ implies $se\leq s$, if $X\scj\lset{s^{-1}s}$, then $sX\scj\lset{s}$. Analogously, if $X\scj\lset{ss^{-1}}$, then $Xs\scj\lset{s}$.
	
	Given $s,t\in S$ and $Y\in\til{\covd}(s)$, we have to show that $tY\in\til{\covd}(ts)$ and $Yt\in\til{\covd}(st)$. First, we suppose that $Y=sX$ for some $X\in\covd(s^{-1}s)$. Since $s^{-1}t^{-1}ts\leq s^{-1}s$, we have that $s^{-1}t^{-1}tsX\in\covd(s^{-1}t^{-1}ts)$, so that $tY=tsX=tss^{-1}t^{-1}tsX\in\til{\covd}(ts)$. Also $t^{-1}Xt\in\covd{D}(t^{-1}s^{-1}st)$, by the condition on $\covd$, so that $Yt=sXt=sXtt^{-1}t=stt^{-1}Xt\in\til{\covd}(st)$, where we used that $X\scj E$ in the last equality. Analogously, if $Y=Xs$ for some $X\in\covd(ss^{-1})$, we show that $tY\in\til{\covd}(ts)$ and $Yt\in\til{\covd}(st)$.
	
	Clearly, if we start with a coverage $\covd$ on $E$ as in the statement, we build $\til{\covd}$ as above and we take the restriction to $E$, we get $\til{\covd}_E=\covd$, since $xe=x=ex$ for every $e\in E$ and $x\in X$ such that $X\in\covd(e)$.
	
	On the other hand, if we start with a coverage $\cov$ on $S$ and build $\til{\cov_E}$ as above, we get, by the definition of coverage, that $\til{\cov_E}(s)\scj\cov(s)$ for every $s\in S$. And for the other inclusion, let $Y\in\cov(s)$. Then, $s^{-1}Y\in\cov(s^{-1}s)=\cov_E(s^{-1}s)$ and $ss^{-1}Y\in\til{\cov_E}(s)$. Also, since $Y\scj\lset{s}$, we get that $ss^{-1}y=y$ for every $y\in Y$ so that $Y=ss^{-1}Y\in \til{\cov_E}(s)$, from where we conclude that $\cov(s)\scj \til{\cov_E}(s)$.
	
	Therefore the map $\covd\mapsto\til{\covd}$ as constructed above is the inverse of the map $\cov\mapsto\cov_{E}$.
\end{proof}

The next goal is to build a pseudogroup from a coverage on a inverse semigroup. For that, we first need a few definitions and results.

\begin{definition}
	Let $S$ be an inverse semigroup and $\cov$ a coverage on $S$. A semigroup homomorphism $\theta:S\to T$ to a pseudogroup $T$ is called a \emph{$\cov$-cover-to-join} map if for every $a\in S$ and $X\in\cov(a)$, we have that \(\theta(a)=\bigvee_{x\in X}\theta(x).\)
\end{definition}

One way of interpreting coverages is that we are imposing join relations on $S$. In other words, we want to find a pseudogroup $P_\cov(S)$ together with a $\cov$-cover-to-join map $\pi:S\to P_\cov(S)$ such that for every $\cov$-cover-to-join map $\theta:S\to T$, there exists a unique pseudogroup homomorphism $\til{\theta}$ such that $\theta=\til{\theta}\pi$. The existence of such pseudogroup in proven in \cite[Theorem 4.20]{MR3077869} for strong coverages that are \emph{idempotent-pure}, that is, $X\in\cov(a)$ and $X\in E(S)$ implies that $a\in E(S)$.

We adapt their proof to show the existence of $P_\cov(S)$ for our definition of coverage, however we have to weaken the universal property. More specifically, we look only at idempotent-pure semigroup homomorphisms between inverse semigroups (see Subsection \ref{subsec:inv.sem}).

In order to find $P_\cov(S)$, we first need the notion of a nucleus on an inverse semigroup and some of its properties.

\begin{definition}\cite[Subsection 4.3]{MR3077869}
	Let $S$ be an inverse semigroup. A map $\nu:S\to S$ is called a nucleus if it satisfies the following conditions:
	\begin{itemize}
		\item[(N1)] $a\leq \nu(a)$ for all $a\in S$.
		\item[(N2)] $a\leq b$ implies that $\nu(a)\leq\nu(b)$.
		\item[(N3)] $\nu^2(a)=\nu(a)$ for all $a\in S$.
		\item[(N4)] $\nu(a)\nu(b)\leq\nu(ab)$ for all $a,b\in S$.
	\end{itemize}
\end{definition}

Given a nucleus $\nu$ on $S$, we define $S_{\nu}=\{a\in S\mid \nu(a)=a\}$ and product on $S_{\nu}$ by $a\cdot b=\nu(ab)$.

\begin{proposition}\cite[Lemma 4.15 and Proposition 4.16]{MR3077869} \label{prop:nucleus}
	Let $\nu$ be a nucleus on a inverse semigroup $S$. Then $(S_\nu,\cdot)$ is an inverse semigroup whose natural partial order coincides with that of $S$ and the map $S\to S_\nu$ given by $a\mapsto\nu(a)$ is a surjective idempotent-pure semigroup homomorphism. Moreover, if $S$ is a pseudogroup then $S_\nu$ is also a pseudogroup and the natural map is a pseudogroup homomorphism.
\end{proposition}

\begin{lemma}\label{lemma:nucleus.preserves.completely.prime}
	Let $\nu$ be a nucleus on a pseudogroup $P$. If $X$ is a completely prime filter in $P$, then $\nu(X)$ is a completely prime filter in $P_\nu$.
\end{lemma}

\begin{proof}
	We start by observing that since $X$ is a filter, condition $(N1)$ implies that $\nu(X)\scj X$, and in particular, $\nu(X)$ is a proper subset of $P_\nu$.
	
	Suppose that $y\in P_\nu$ is such that $\nu(x)\leq y$ for some $x\in X$. Since $\nu(x)\in X$, we have that $y\in X$ so that $y=\nu(y)\in\nu(X)$.
	
	Given $\nu(x),\nu(y)\in\nu(X)$, where $x,y\in X$, there exists $z\in X$ such that $z\leq x$ and $z\leq y$. Then $\nu(z)\in\nu(X)$ and, by (N2), $\nu(z)\leq\nu(x)$ and $\nu(z)\leq \nu(y)$.
	
	Finally, let $I$ be a compatible set on $P_\nu$ such that $\bigvee I\in \nu(X)$. By Proposition \ref{prop:nucleus}, $I$ is a compatible set on $P$. Observing that $\bigvee I\in X$, we find $y\in I$ such that $y\in X$. It follows that $y=\nu(y)\in\nu(X)$.
\end{proof}

Now, let $\cov$ be a coverage on a inverse semigroup $S$. We will use $\cov$ to define a nucleus on the inverse semigroup $C(S)$ of all compatible order ideal of $S$ with operation given by subset multiplication \cite{MR0325820}.

A subset $A\scj S$ is said to be \emph{$\cov$-closed} if $X\scj A$ and $X\in\cov(a)$ implies that $a\in A$. We define a map $\nu:C(S)\to C(S)$ by $\nu(A)$ as the intersection of all $\cov$-closed compatible ordered ideals that contains $A$.

\begin{proposition}\label{prop:nucleus.from.cov}
	The map $\nu:C(S)\to C(S)$ define above is a nucleus.
\end{proposition}

\begin{proof}
	The proof is the same as in the first part of \cite[Theorem 4.20]{MR3077869} since \cite[Lemma 4.19]{MR3077869} still holds with our weaker definition of coverage.
\end{proof}

Define $P_{\cov}(S)$ to be $(C(S)_{\nu},\cdot)$, which is a pseudogroup by \cite[Theorem 1.15]{MR0325820} and Proposition \ref{prop:nucleus}. Also define the map $\pi:S\to C(S)_{\nu}$ by $\pi(a)=\nu(\lset{a})$.

\begin{remark}\label{rem:universal.pseudogroup}
	The set $P_{\cov}(S)$ is exactly the set of all $\cov$-closed compatible ordered ideals of $S$. Moreover, if $I\in P_{\cov}(S)$, then $\bigvee_{a\in I}\pi(a)=I$, since $I$ is the smallest $\cov$-closed compatible ordered ideals of $S$ that contains every element of $I$, and in particular, $I$ contains $\pi(a)$ for all $a\in I$.
\end{remark}

\begin{theorem}\label{thm:univesal.property}
	The map $\pi:S\to P_{\cov}(S)$ is a $\cov$-cover-to-join idempotent-pure map. Moreover for every $\cov$-cover-to-join idempotent-pure map $\theta:S\to T$, there exists a unique idempotent-pure pseudogroup homomorphism $\til{\theta}:P_{\cov}(S)\to T$ such that $\theta=\til{\theta}\pi$.
\end{theorem}

\begin{proof}
	That $\pi$ is a semigroup homomorphism follows from \cite[Lemma 1.10]{MR0325820}. By \cite[Lemma  of Theorem 1.15]{MR0325820}, if $A\in E(C(S))$, then $A\scj E(S)$ so that $\pi$ is idempotent-pure. Now, let $a\in S$ and $X\in\cov(a)$. By (N2), we have that $\bigvee_{x\in X}\pi(x)\leq \pi(a)$. On the other hand $\bigvee_{x\in X}\pi(x)$ is a $\cov$-closed ordered ideal than contains $X$, so that $a\in \bigvee_{x\in X}\pi(x)$, and hence $\pi(a)\leq \bigvee_{x\in X}\pi(x)$.
	
	Using the universal property of $C(S)$ \cite[Theorem 1.15]{MR0325820}, the map $\bar{\theta}:C(S)\to T$ given by $\bar{\theta}(A)=\bigvee \theta(A)$ is the unique pseudogroup homomorphism such that $\bar{\theta}(\lset{s})=\theta(s)$. We show that $\bar{\theta}$ factors through $\nu$, and for that it is sufficient to prove that $\bar{\theta}(A)=\bar{\theta}(\nu(A))$ for $A\in C(S)$. By (N1), $\bar{\theta}(A)\leq \bar{\theta}(\nu(A))$. Now consider the set $I=\{a\in S\mid \theta(a)\leq\bar{\theta}(A)\}$, which is an ordered ideal containing $A$. Given $a,b\in I$, we have that $\theta(ab^{-1})\leq \bar{\theta}(AA^{-1})\in E(T)$. Since $\theta$ is idempotent-pure $ab^{-1}\in E(S)$ and, analogously, $a^{-1}b\in E(S)$. Hence $I$ is compatible. Suppose now that $X\in\cov(a)$ and $X\scj I$. Using that $\theta$ is a $\cov$-cover-to-join map, we see that $a\in I$. We conclude that $I$ is a $\cov$-closed compatible ordered ideal containing $A$ so that $\nu(A)\scj I$. By the definition of $I$, $\bar{\theta}(\nu(A))\leq \bar{\theta}(A)$. The map $\til{\theta}$ is then given by $\til{\theta}(\nu(A))=\bar{\theta}(A)$. Finally, if $\til{\theta}(\nu(A))\in E(T)$, then $a\in E(S)$ for all $a\in A$, since $\theta$ is idempotent-pure, and therefore $\nu(A)$ is an idempotent of $P_{\cov}(S)$.
\end{proof}

\begin{remark}
	The proof of Theorem \ref{thm:univesal.property} is adapted from Lawson and Lenz's proof of \cite[Theorem 4.20]{MR3077869}. We point out the main differences. The first is in the definition of the nucleus $\nu$. In \cite{MR3077869} they define a $\cov$-closure $\overline{A}$ of a set $A\scj S$ and show that it is the intersection of all $\cov$-closed ideals containing, however this proof heavily relies on properties (MS) and (T) of a strong coverage (see \cite[Lemma 4.18]{MR3077869} and the paragraph preceding it). They then define a nucleus on $C(S)$ by $\nu(I)=\overline{I}$. In a way, we have skipped the step of defining the $\cov$-closure. The second difference is that our universal property is weaker than theirs since we are only dealing with idempotent-pure semigroup homomorphisms.
\end{remark}

\begin{remark}
	In Theorem \ref{thm:univesal.property}, if $S$ consists only of idempotents, it can be seen as semilattice. In this case $P_\cov(S)$ is the frame define by Johnstone in \cite[Section II.2]{MR698074} from a coverage on a semilattice. In this case, we will denote $P_\cov(S)$ by $F_\cov(S)$.
\end{remark}

We end this section with two lemmas that will be needed throughout this paper. The first generalizes \cite[Lemma 2.2]{de_castro_2020}. The second can be thought as imposing relations on a pseudogroup given a map from an inverse semigroup and a coverage on it.

\begin{lemma}\label{lemma:semigroup.generates.pseudogroup}
	Let $P$ be a pseudogroup and let $S\scj P$ be a inverse subsemigroup. Suppose that $S$ generates $P$ in the sense that for all $p\in P$, there exists a compatible set $X\scj S$ such that $p=\bigvee X$. Define a coverage on $S$ by $X\in\cov'(a)$ for $a\in S$ if $X\scj S$ is a compatible set and $a=\bigvee X$ in $P$. Then $P\cong P_{\cov'}(S)$ as pseudogroups.
\end{lemma}

\begin{proof}
	By Remark \ref{rem:universal.pseudogroup} and Theorem \ref{thm:univesal.property} (the inclusion is clearly idempotent-pure) the map $\phi:P_{\cov'}(S)\to P$ given by $\phi(I)=\bigvee I$ is a pseudogroup homomorphism that is surjetcive because $S$ generates $P$.
	
	We now prove that $\phi$ is injective. Let $I\in P_{\cov'}(S)$, $p=\bigvee I$ and $J=\{x\in S\mid x\leq p\}$. We prove that $I=J$ so that $I$ is uniquely determined by its image. Clearly $I\scj J$. Let $x\in S$ such that $x\leq p$, that is $x=px^{-1}x$. Since the product of $P$ distributes over joins, we have that $x=\bigvee_{a\in I}ax^{-1}x$. Note that $x\in S$ and $I\scj S$, so that $Ix^{-1}x$ is a covering of $x$, which implies that $x\in I$ because $I$ is $\cov'$-closed. Hence $I=J$ and $\phi$ is injective.	
\end{proof}

\begin{lemma}\label{lem:from.cov.inv.semigroup.to.cov.pseudogroup}
	Let $\cov$ be a coverage on an inverse semigroup $S$, $P$ be a pseudogroup and $\theta:S\to P$ an idempotent-pure semigroup homomorphism. For each $p\in P$, define $\covd(p)$ to be the set all compatible subsets $I$ of $P$ such that $\bigvee I = p$, all sets of the form $\theta(X)$, $q\theta(X)$, $\theta(X)r$ or $q\theta(X)r$ for $X\in\cov(a)$, whenever $p=q(a)$, $p=q\theta(a)$, $p=\theta(a)r$ or $p=q\theta(a)r$ respectively, where $a\in S$ and $q,r\in P$. Then $\covd$ is a coverage on $P$ and if $\pi:P\to P_{\covd}(P)$ is the map from the universal property for $\covd$, then $\pi\circ\theta:S\to P_{\covd}(P)$ is a $\cov$-cover-to-join map.
\end{lemma}

\begin{proof}
	The proof that $\covd$ is a coverage is straightforward. For example, if $Y$ is a covering of $p$ the form $q\theta(X)$ where $X$ is a covering of $a$ and $p=q\theta(a)$, then for any $r\in P$, we have that $rY=rq\theta(X)$ is a covering of $rp$ since $rp=rq\theta(a)$.
	
	We now prove that $\pi\circ\theta$ is a $\cov$-cover-to-join map. For $a\in S$ and $X\in\cov(a)$, we have that $\theta(X)\in\covd(\theta(a))$. Since $\pi$ is a $\covd$-cover-to-join map, we get
	\[\pi(\theta(a))=\bigvee_{q\in\theta(X)}\pi(q)=\bigvee_{b\in X}\pi(\theta(b)).\]
\end{proof}

With the condition of the above lemma, we will say that $\covd$ is the \emph{coverage induced by $\theta$ and $\cov$}.


\section{Nuclei and étale groupoids associated to pseudogroups}\label{sec:nucleus}

Given a nucleus $\nu:P\to P$ on a pseudogroup $P$, we have seen in the last section that the fixed points of $\nu$ denoted by $P_{\nu}$ is a pseudogroup with multiplication given by $a\cdot b=\nu(ab)$ and that $\nu$ can be seen as pseudogroup homomorphism between $P$ and $(P_\nu,\cdot)$. The main goal of this section is to show that $\nu$ induces a natural map between to the étale groupoids associated to $P$ and $P_\nu$. In order to do that, we work with the notion of germs (in a abstract sense) as defined by Resende in \cite{resende2006lectures}.

\begin{definition}
	Let $S$ be an inverse semigroup and $F\scj E(S)$ a filter in $E(S)$. Let $s\in S$ be such that $ss^{-1}\in F$. The germ of $s$ at $F$ is the set
	\[\germ_F\, s=\{t\in S\mid tt^{-1}\in F\text{ and }ft=fs\text{ for some }f\in F\}.\]
\end{definition}

Part of the following lemma is an exercise in \cite{resende2006lectures}. Since the results will be needed here, we provide the proof.

\begin{lemma}\label{lemma:germs}
	Let $S$ be an inverse semigroup, $F\scj E(S)$ be a filter in $E(S)$ and $s\in S$ be such that $ss^{-1}\in F$. Then:
	\begin{enumerate}[(i)]
		\item\label{lemma:germs.fs} $fs\in \germ_F\, s$ for all $f\in F$;
		\item\label{lemma:germs.F} $F$ coincides with the set $\usetr{\{tt^{-1}\mid t\in\germ_F\, s\}}{E(S)}$;
		\item\label{lemma:germs.filter} $\germ_F\, s$ is a filter in $S$;
		\item\label{lemma:germs.independence} if $r\in\germ_F\, s$, then $\germ_F\, r=\germ_F\, s$;
		\item\label{lemma:germs.completely.prime} if $S$ is also a pseudogroup and $F$ is completely prime on $E(S)$, then $\germ_F\, s$ is completely prime on $S$.
	\end{enumerate}
\end{lemma}

\begin{proof}
	(\ref{lemma:germs.fs}) Since $F$ is a filter, we have that $F$ is closed under multiplication. Given $f\in F$, we have that $fs(fs)^{-1}=fss^{-1}f=fss^{-1}\in F$ and $ffs=fs$, so that $fs\in \germ_F\, s$.
	
	(\ref{lemma:germs.F}) Given $f\in F$, by (\ref{lemma:germs.fs}) $fs\in \germ_F\, s$, and since $f\geq fss^{-1}=fs(fs)^{-1}$, we have that $f\in \usetr{\{tt^{-1}\mid t\in\germ_F\, s\}}{E(S)}$. Now suppose that $f\in \usetr{\{tt^{-1}\mid t\in\germ_F\, s\}}{E(S)}$, then there exists $t\in \germ_F\, s$ such that $f\geq tt^{-1}$. Since $tt^{-1}\in F$ and $F$ is a filter, we have that $f\in F$.
	
	(\ref{lemma:germs.filter}) By (\ref{lemma:germs.F}), $\germ_F\, s$ is a proper subset of $S$, otherwise $F=E(S)$. Let $t\in \germ_F\, s$ and let $u\in S$ be such that $t\leq u$. Since $F$ is a filter, $tt^{-1}\in F$ and $tt^{-1}\leq uu^{-1}$, we have that $uu^{-1}\in F$. Also there exists $f\in F$ such that $fs=ft$, so that $ftt^{-1}\in F$ and $ftt^{-1}s=tt^{-1}fs=tt^{-1}ft=ftt^{-1}t=ft=ftt^{-1}u$. Hence, $u\in \germ_F\, s$.
	
	Now, fix $t,u\in\germ_F\, s$. There exists $e,f\in F$ such that $et=es$ and $fu=fs$. By (\ref{lemma:germs.fs}), $efs\in \germ_F\, s$. Also $efs=fes=fet\leq t$ and $efs=efu\leq u$. If follows that $\germ_F\, s$ is a filter.
	
	(\ref{lemma:germs.independence}) Suppose that $r\in \germ_F\, s$, then $rr^{-1}\in F$ and there exists $e\in F$ such that $er=es$. Now, if $t\in \germ_F\, s$, $tt^{-1}\in F$ and there exists $f\in F$ such that $fs=ft$. Since $F$ is a filter in $E(S)$, $ef\in F$ and $efr=fer=fes=efs=eft$. This implies that $t\in \germ_F\, r$. The other inclusion is analogous.
	
	(\ref{lemma:germs.completely.prime}) Let $A$ be a compatible subset of $S$ and suppose that $t:=\bigvee A\in \germ_F\, s$. Since joins distribute over multiplication, $tt^{-1}=\bigvee_{u,w\in A}uw^{-1}\in F$. Using that $F$ is completely prime and that $A$ is compatible, we have that there exist $u,w\in A$ such that $uw^{-1}\in F$. We claim that $w\in \germ_F\, s$, which proves that $\germ_F\, s$ is completely prime. Since $uw^{-1}$ is idempotent and $\germ_F\, s$ is a filter, it is sufficient to show that $uw^{-1}w\in \germ_F\, s$. First, $uw^{-1}w(uw^{-1}w)^{-1}=uw^{-1}wu^{-1}=uw^{-1}(uw^{-1})=uw^{-1}\in F$. Now, there exists $f\in F$ such that $ft=fs$. Also, $ww^{-1}t=w$ because $w\leq t$. Then $fuw^{-1}\in F$ and $fuw^{-1}uw^{-1}w=fuw^{-1}w=fuw^{-1}ww^{-1}t=fuw^{-1}t=uw^{-1}ft=uw^{-1}fs=fuw^{-1}s$. It follows that $uw^{-1}w\in\germ_F\, s$, completing the proof.
\end{proof}

\begin{remark}
	The previous Lemma shows that $\germ_F\, s$ is a filter in $S$. On the other hand, if $A$ is filter in $S$, if we define $F=\usetr{AA^{-1}}{E(S)}$ we have that $F$ is a filter in $E(S)$ and for any $a\in A$, we have that $\germ_F\,a=\usetr{Fa}{S}=A$ (see \cite[Subsection 3.3]{MR3109745}).
\end{remark}

Next lemma shows how we can use germs in order to work with filters and nuclei.

\begin{lemma}\label{lemma:filter.doesnt.depend.on.a}
	Let $\nu:S\to S$ be a nucleus on $S$, and $A$ be a filter in $S_{\nu}$. Define $F_A=\nu^{-1}((\usetr{A\cdot A^{-1}}{S_\nu})\cap E(S_{\nu}))$ and let $a\in A$ be given. Then:
	\begin{enumerate}[(i)]
		\item\label{lemma:filter.doesnt.depend.on.a.F_A} $F_A$ is a filter in $E(S)$;
		\item\label{lemma:filter.doesnt.depend.on.a.germ} $\germ_{F_A}\, a$ is a filter that does not depend on $a$;
		\item\label{lemma:filter.doesnt.depend.on.a.pertinence} for all $x\in S$, we have that $x\in\germ_{F_A}\, a$ if and only if $\nu(x)\in\germ_{F_A}\, a$;
		\item\label{lemma:germ.equal.preimage} $\germ_{F_A}\, a=\nu^{-1}(A)$;
		\item\label{lemma:filter.doesnt.depend.on.a.pseudogroup} if $S$ is a pseudogroup and $A$ is completely prime, then $\germ_{F_A}\, a$ is completely prime.
		
	\end{enumerate}
\end{lemma}

\begin{proof}
	\eqref{lemma:filter.doesnt.depend.on.a.F_A} That $(\usetr{A\cdot A^{-1}}{S_\nu})\cap E(S_{\nu})$ is a filter in $E(S_{\nu})$ follows from \cite[Lemma 3.5]{MR3109745}. Since $\nu$ is idempotent-pure, the meets on $E(S)$ and $E(S_{\nu})$ are given by multiplication and $\nu$ preserves multiplication when thought as map from $S$ to $S_{\nu}$, we have that $F_A$ is a filter in $E(S)$.
	
	\eqref{lemma:filter.doesnt.depend.on.a.germ} By Lemma \ref{lemma:germs}(\ref{lemma:germs.filter}), $\germ_{F_A}\, a$ is a filter in $S$.
	
	Given $b\in A$, we show that $\germ_{F_A}\, a=\germ_{F_A}\, b$. Since $A$ is a filter, there exists $t\in A$ such that $t\leq a$ and $t\leq b$, or equivalently, $t=tt^{-1}a=tt^{-1}b$ (the order on $(S_{\nu},\cdot)$ coincides with the one on $S$). Given $s\in\germ_{F_A}\, a$, we have that $ss^{-1}\in F_A$ and there exists $f\in F_A$ such that $fs=fa$. Notice that $tt^{-1}\in F_A$, since $\nu$ is the identity map on $S_{\nu}$ and $\nu(tt^{-1})=t\cdot t^{-1}$. Since $F_A$ is a filter in $E(S)$, we have that $tt^{-1}f\in F_A$. Also,
	\[tt^{-1}fs=tt^{-1}fa=ftt^{-1}a=ftt^{-1}b=tt^{-1}fb,\]
	which implies that $s\in\germ_{F_A}\, b$. This proves the inclusion $\germ_{F_A}\, a\scj\germ_{F_A}\, b$. The reverse inclusion is analogous and hence $\germ_{F_A}\, a=\germ_{F_A}\, b$.
	
	\eqref{lemma:filter.doesnt.depend.on.a.pertinence} For $x\in S$, if $x\in\germ_{F_A}\, a$, since $x\leq\nu(x)$ and $\germ_{F_A}\, a$ is a filter, we have that $\nu(x)\in\germ_{F_A}\, a$. On the other hand, if $\nu(x)\in\germ_{F_A}\, a$, then $\nu(xx^{-1})=\nu(x)\cdot\nu(x)^{-1}=\nu(\nu(x)\nu(x)^{-1})\in (\usetr{A\cdot A^{-1}}{S_\nu})\cap E(S_{\nu})$, so that $xx^{-1}\in F_A$. Since $x\leq \nu(x)$, we have that $x=xx^{-1}x\leq xx^{-1}\nu(x)\in \germ_{F_A}\, a$.
	
	\eqref{lemma:germ.equal.preimage} Given $x\in S$, suppose first that $x\in\germ_{F_A}\, a$. In this case, there exists $e\in F_A$ such that $ex=ea$. It follows that $\nu(ex)=\nu(ea)=\nu(e)\cdot a\in \usetr{A\cdot A^{-1}\cdot A}{S_\nu}=A*A^{-1}*A=A$ (see Subsection \ref{subsec:groupoid}). Since $\nu$ preservers the order, $ex\leq x$ and $A$ is a filter, we have that $x\in\nu^{-1}(A)$. Now suppose that $x\in\nu^{-1}(A)$. Then $\nu(x)\in A\scj \germ_{F_A}\, a$, where the inclusion comes from \eqref{lemma:filter.doesnt.depend.on.a.germ}. By \eqref{lemma:filter.doesnt.depend.on.a.pertinence}, $x\in  \germ_{F_A}\, a$.
	
	\eqref{lemma:filter.doesnt.depend.on.a.pseudogroup} Suppose that $S$ is a pseudogroup. From \cite[Lemma 2.2]{MR3077869}, we have that $\usetr{A\cdot A^{-1}}{S_\nu}$ is completely prime. Since $\nu$ is a pseudogroup homomorphism, when restricted to the idempotents, it is a frame homomorphism, so that $F_A$ is completely prime. Then $\germ_{F_A}\, a$ is completely prime by Lemma \ref{lemma:germs}(\ref{lemma:germs.completely.prime}).
\end{proof}

For the next theorem we will use the notations introduced in subsection \ref{subsec:groupoid}. In particular, we will use the same notation for the product on both $\gr{P_{\nu}}$ and $\gr{P}$, the context making it clear which of the groupoid we are dealing with.

\begin{theorem}\label{thm:nucleus.pseudogroup.hom.grp}
	Let $\nu:P\to P$ be a nucleus on a pseudogroup $P$. The map $\Phi:\gr{P_{\nu}}\to\gr{P}$ given by $\Phi(A)=\nu^{-1}(A)$ is an embedding of topological groupoids. Moreover, for $X\in\gr{P}$, we have that $X\in\Image(\Phi)$ if and only if $\ran{X}\in\Image(\Phi)$.
\end{theorem}

\begin{proof}
	By Lemma \ref{lemma:filter.doesnt.depend.on.a}, the map $\Phi$ is well defined.
	
	The following observations are used throughout the proof. Since $\nu$ is idempotent-pure, for every $A\in\gr{P_\nu}$, the elements of $F_A$ are idempotents on $P$. By Lemma \ref{lemma:filter.doesnt.depend.on.a}, for $x\in P$, $x\in\Phi(A)$ if and only if $\nu(x)\in\Phi(A)$. Also, for $A\in\gr{P_\nu}$, we define $F_A=\nu^{-1}((\usetr{A\cdot A^{-1}}{P_\nu})\cap E(P_{\nu}))$, so that $\Phi(A)=\germ_{F_A}\,a$ for any $a\in A$.
	
	Notice that for $A\in \gr{P_\nu}$, we have that $A=\Phi(A)\cap P_\nu$. The injectivity of $\Phi$ follows immediately.
	
	For the multiplicativity of $\Phi$, we first show that, given $A,B\in\gr{P_{\nu}}$, we have that $\dom{A}=\ran{B}$ implies that $\dom{\Phi(A)}=\ran{\Phi(B)}$. Suppose that, $s\in\dom{\Phi(A)}$ so that there exist $x,y\in\Phi(A)$ such that $x^{-1}y\leq s$. In order to work with $\Phi(A)$ and $\Phi(B)$ we fix arbitrary elements $a\in A$ and $b\in B$. Since $\Phi(A)=\germ_{F_A}\,a$, there exist $e_x,e_y\in F_A$ such that $e_xa=e_xx$ and $e_ya=e_yy$. We claim that $z:=bb^{-1}a^{-1}e_xe_ya\in F_B$. Clearly $z$ is an idempotent so that $\nu(z)\in E(P_{\nu})$. And, since $e_x,e_y\in F_A$, we have that $\nu(e_x),\nu(e_y)\in\usetr{A\cdot A^{-1}}{P_{\nu}}=A*A^{-1}$, so that
	\[\nu(z)=\nu(bb^{-1}a^{-1}e_xe_ya)=b\cdot b^{-1}\cdot a^{-1}\cdot\nu(e_x)\cdot\nu(e_y)\cdot a\in B*B^{-1}*A^{-1}*A*A^{-1}*A*A^{-1}*A\]
	\[=B*B^{-1}*A^{-1}*A=B*B^{-1}*B*B^{-1}=B*B^{-1}=\usetr{B\cdot B^{-1}}{P_\nu}.\]
	By Lemma \ref{lemma:germs}(\ref{lemma:germs.fs}), $zb\in \germ_{F_B}\,b=\Phi(B)$. Now,
	\[zb(zb)^{-1}=zbb^{-1}z=bb^{-1}a^{-1}e_xe_yabb^{-1}bb^{-1}a^{-1}e_xe_ya=bb^{-1}a^{-1}e_xe_ya\]
	\[=bb^{-1}x^{-1}e_xe_yy\leq x^{-1}e_xe_yy\leq x^{-1}y\leq s,\]
	which implies that $s\in\ran{\Phi(B)}$. For the other inclusion, take $r\in \ran{\Phi(B)}$, $u,v\in\Phi(B)$ such that $uv^{-1}\leq r$ and $f_u,f_v\in F_B$ such that $f_ub=f_uu$ and $f_vb=f_vv$. A similar argument shows that $w:=af_ubb^{-1}f_va^{-1}a\in\Phi(A)$ and that $w^{-1}w\leq r$, so that $r\in\dom{\Phi(A)}$.
	
	We now show that if $A,B\in\gr{P_{\nu}}$ are such that $A^{-1}*A=\dom{A}=\ran{B}=B*B^{-1}$, then $\Phi(A)*\Phi(B)=\Phi(A*B)$. In order to compute $\Phi(A*B)$, observe that $F_{A*B}=\nu^{-1}((\usetr{(A*B)\cdot (A*B)^{-1}}{P_\nu})\cap E(P_{\nu}))=\nu^{-1}((A*B*B^{-1}*A^{-1})\cap E(P_{\nu}))$. Fix $a\in A$, $b\in B$ and notice that $\nu(ab)=a\cdot b\in\Phi(A*B)$, and therefore $ab\in\Phi(A*B)$.
	By Lemma \ref{lemma:germs}(\ref{lemma:germs.independence}), we can use $ab$ to compute $\Phi(A*B)$, that is, $\Phi(A*B)=\germ_{F_{A*B}} ab$.
	
	For the inclusion $\Phi(A)*\Phi(B)\scj\Phi(A*B)$, it is sufficient to prove that $\Phi(A)\Phi(B)\scj\Phi(A*B)$, since $\Phi(A*B)$ is a filter. Take $x\in \Phi(A)$, $y\in\Phi(B)$, so that $xx^{-1}\in F_A$, $yy^{-1}\in F_B$ and there exist $e\in F_A$, $f\in F_B$ such that $ex=ea$ and $fy=fb$. By a previous observation, $\nu(x)\in A$ and $\nu(y)\in B$. Define $g=xfx^{-1}e$. We have that
	\[\nu(g)=\nu(x)\cdot\nu(f)\cdot\nu(x)^{-1}\cdot\nu(e)\in A*B*B^{-1}*A^{-1}*A*A^{-1}=A*B*B^{-1}*A^{-1},\]
	and $\nu(g)\in E(P_{\nu})$ since $\nu(e),\nu(f)\in E(P_{\nu})$. Also
	\[\nu(xyy^{-1}x^{-1})=\nu(x)\cdot\nu(y)\cdot\nu(y)^{-1}\cdot\nu(x)^{-1}\in A*B*B^{-1}*A^{-1}\cap E(P_{\nu}).\]
	And finally,
	\[gxy=xfx^{-1}exy=xx^{-1}exfy=xx^{-1}exfb=xfx^{-1}exb=xfx^{-1}eab=gab.\]
	We conclude that $xy\in\Phi(A*B)$.
	
	For the other inclusion $\Phi(A*B)\scj \Phi(A)*\Phi(B)$, take $z\in\Phi(A*B)$ and $g\in F_{A*B}$ such that $gz=gab$. Define $y=a^{-1}gab$, and notice that $y\in \Phi(B)$, by Lemma \ref{lemma:germs}(\ref{lemma:germs.fs}), since $a^{-1}ga$ is idempotent and
	\[\nu(a^{-1}ga)=a^{-1}\cdot\nu(g)\cdot a\in A^{-1}*A*B*B^{-1}*A^{-1}*A=(B*B^{-1})^3=B*B^{-1}.\]
	Also,
	\[ay=aa^{-1}gab=gab=gz\leq z,\]
	which implies that $z\in\Phi(A)*\Phi(B)$.
	
	The continuity of $\Phi$ follows from the equality
	\[\Phi^{-1}(\{C\in \gr{P}\mid x\in C\})=\{A\in\gr{P_{\nu}}\mid \nu(x)\in A\},\]
	for every $x\in P$. And when we restrict the codomain of $\Phi$ to its image, we have that $\Phi$ is an open map since
	\[\Phi(\{A\in\gr{P_{\nu}}\mid a\in A\})=\{C\in \Phi(\gr{P_\nu})\mid a\in C\},\]
	for every $a\in P_\nu$.
	
	For the last part, let $X\in\gr{P}$. If $X=\Phi(A)$ for some $A\in\gr{P_\nu}$, then $\ran{X}=\Phi(\ran{A})\in\Image(\Phi)$. Now, suppose that $X*X^{-1}=\ran{X}=\Phi(B)$ for some $B\in\gr{P_\nu}$. By Lemma \ref{lemma:nucleus.preserves.completely.prime}, $\nu(X)\in\gr{P_\nu}$. We prove that $X=\Phi(\nu(X))=\germ_{F_{\nu(X)}} a$, where $a\in\nu(X)$ is a fixed element. Suppose first that $b\in X$. Then $\nu(bb^{-1})=\nu(b)\cdot \nu(b)^{-1}\in \nu(X)\cdot\nu(X)^{-1}\cap E(P_{\nu})$, so that $bb^{-1}\in F_{\nu(X)}$. Also, since $b\leq \nu(b)$, we have that $bb^{-1}\nu(b)=b=bb^{-1}b$, so that $b\in\germ_{F_{\nu(X)}} \nu(b)=\germ_{F_{\nu(X)}} a$, where the last equality follows from Lemma \ref{lemma:filter.doesnt.depend.on.a}. Now suppose that $b\in \germ_{F_{\nu(X)}} a$. Then, there exists $e\in F_{\nu(X)}$ such that $eb=ea$. By the definition of $F_{\nu(X)}$, there exists a pair of compatible elements $r,s\in X$ such that $\nu(e)\geq \nu(r)\cdot\nu(s)^{-1}=\nu(rs^{-1})$. Observe that $\nu(rs^{-1})\in X*X^{-1}=\Phi(B)$, and hence, $\nu(e)\in\Phi(B)$. By the first part of the proof, $e\in\Phi(B)=X*X^{-1}$, so that $eb=ea\in X*X^{-1}*X=X$. Finally, $eb\leq b$ implies that $b\in X$.
\end{proof}

\begin{remark}
	The proof of Theorem \ref{thm:nucleus.pseudogroup.hom.grp} works practically the same if we replace the pseudogroup $P$ for an inverse semigroup $S$ and the groupoids $\gr{P}$ and $\gr{P_\nu}$ for $\grl(S)$ and $\grl(S_\nu)$ respectively.
\end{remark}


\section{Imposing join relations on bisections of an étale groupoid}\label{sec:imposing}

Let $(X,\tau)$ be a sober topological space with a basis $B$ that is closed under intersection and let $\cov$ be a coverage on $B$. As done by the author in \cite[Section 2]{de_castro_2020}, we can define a nucleus $\nu$ on $\tau$ and using the map $\eta$ defined in Subsection \ref{subsec:groupoid}, we define a subspace $X_\cov=\{\eta^{-1}(\nu^{-1}(A))\in X\mid A\in \Sp{\tau_\nu}\}$. Notice that, due to Theorem \ref{thm:nucleus.pseudogroup.hom.grp}, we have that $X_\cov$ is homeomorphic to $\Sp{\tau_\nu}$. If moreover the space $X$ is $T_1$, we can use \cite[Lemma 2.2]{de_castro_2020} to describe $X_\cov$. Since this description will be need in the next section as well, we write it as a lemma.

\begin{lemma}\cite[Lemma 2.2]{de_castro_2020}\label{lem:subspace.from.coverage}
	Let $(X,\tau)$ be a $T_1$ sober topological space with a basis $B$ that is closed under intersection, let $\cov$ be a coverage on $B$ and $\nu:\tau\to\tau$ the nucleus defined from the coverage $\cov$. Then, $\Sp{\tau_\nu}$ can be identified withe the the subspace $X_\cov$ of all points $x\in X$ with the following property: if $Z\in\cov(U)$ for some $U\in B$ and $x\notin V$ for all $V\in Z$, then $x\notin U$.
\end{lemma}

The goal in this section is to generalize this construction replacing $X$ for an étale groupoid $\grp$ and $B$ for a generating inverse subsemigroup of $\bis(\grp)$. Most of the needed results were already proved in this paper. We need one more lemma.

\begin{lemma}\label{lemma:restiction.cov}
	Let $\cov$ be a coverage on a inverse semigroup $S$, and $\cov_{E}$ the restriction to its idempotents semilattice $E:=E(S)$. Then:
	\begin{enumerate}[(i)]
		\item\label{lemma:restriction.cov.idempotents} $E(P_\cov(S))=F_{\cov_E}(E)$;
		\item\label{lemma:restiction.cov.unit.space} The map $\Phi:\gr{P_\cov(S)}^{(0)}\to \Sp{F_{\cov_E}(E)}$ given by $\Phi(A)=A\cap F_{\cov_E}(E)$ is a homeomorphism.
	\end{enumerate}
\end{lemma}

\begin{proof}
	(\ref{lemma:restriction.cov.idempotents}) By definition, $P_\cov(S)$ is the pseudogroup given by the nucleus $\nu:C(S)\to C(S)$ in Proposition \ref{prop:nucleus.from.cov}. By an unnumbered lemma after \cite[Theorem 1.15]{MR0325820}, the idempotents of $C(S)$ are subsets of $E$ and $E(C(S))=C(E)$. Since $\nu$ is idempotent-pure, we can restrict $\nu$ to $C(E)$ both in the domain and codomain, which gives the same nucleus defined in \cite{MR698074} and therefore $E(P_\cov(S))=\nu(E(C(S)))=\nu(C(E))=F_{\cov_E}(E)$.
	
	(\ref{lemma:restiction.cov.unit.space}) As discussed in \cite[Section 2]{MR3077869}, the elements of $\gr{P_\cov(S)}^{(0)}$ are the completely prime filters on $P_\cov(S)$ that contains idempotents, so $\Phi$ is well-defined. For the inverse, we set $\Psi:\Sp{F_{\cov_E}(E)}\to \gr{P_\cov(S)}^{(0)}$ as $\Psi(B)=\usetr{B}{P_\cov(S)}$. It is clear that $\usetr{B}{P_\cov(S)}$ is a filter that contains idempotents. We need to check that $\usetr{B}{P_\cov(S)}$ is completely prime. For that, let $X$ be a compatible subset of $P_\cov(S)$ and suppose that $\bigvee X\in \usetr{B}{P_\cov(S)}$. There exists $e\in B$ such that $e\leq \bigvee X$, so that $e=e\bigvee X=\bigvee_{x\in X}ex\in B$. Since $B$ is completely prime, there exists $x\in X$ such that $ex\in B$. Because $ex\leq x$, we conclude that $x\in\usetr{B}{P_\cov(S)}$, and hence $\Psi$ is well-defined.
	
	Straightforward computations show that $\Phi(\Psi(B))=B$ for all $B\in \Sp{F_{\cov_E}(E)}$ and that $\Psi(\Phi(A))\scj A$ for all $A\in \gr{P_\cov(S)}^{(0)}$. We check the inclusion $A\scj \Psi(\Phi(A))$ for $A\in \gr{P_\cov(S)}^{(0)}$. Take $a\in A$ and $e\in E(P_\cov(S))\cap A$ (which exists by the description of $\gr{P_\cov(S)}^{(0)}$). Since $A$ is a filter, we have that $e\wedge a\in A$. Also, $e\wedge a\in E(P_\cov(S))$ because $e\wedge a\leq e$. Hence $e\wedge a\in \Phi(A)$, from where it follows that $a\in\Psi(\Phi(A))$.
	
	Given $A\in \gr{P_\cov(S)}^{(0)}$ and $a\in A$, we saw above that there exists $f\in E(P_\cov(S))\cap A$ such that $f\leq a$. This implies that the family $\{A\in\gr{P_\cov(S)}^{(0)}\mid f\in A\}_{f\in E(P_\cov(S))}$ is a basis for the topology on $\gr{P_\cov(S)}^{(0)}$. The continuity of $\Phi$ and $\Psi$ then follow from the observation that for $f\in E(P_\cov(S))$ and $B\in \Sp{F_{\cov_E}(E)}$, we have that $f\in B$ if and only if $f\in \Phi(B)$.
\end{proof}

For the next result, we need a little bit of set-up. Let $\grp$ be a sober étale groupoid, let $\bis(\grp)$ be the pseudogroup of bisections and suppose the $S\scj \bis(\grp)$ is a inverse subsemigroup that generates $\bis(\grp)$. Let $\cov$ be a coverage on $S$, $\cov'$ the coverage given by Lemma \ref{lemma:semigroup.generates.pseudogroup} and $\overline{\cov}=\cov\cup\cov'$. The set of idempotents of $\bis(\grp)$, denoted by $E$, is exactly the subspace topology on $\grp^{(0)}$ and $D:=S\cap E$ is a basis for this topology. If we restrict the coverages $\cov,\cov',\overline{\cov}$ we get coverages $\cov_D,\cov'_D,\overline{\cov}_D$ on $D$. From the discussion at the beginning of this section, we get a subspace of $\grp^{(0)}$, which we will denote by, $\grp^{(0)}_{\cov}$. As observed at the beginning of the section, $\grp^{(0)}_{\cov}$ is homeomorphic to $\Sp{F_{\overline{\cov}_D}(D)}$.

\begin{proposition}\label{prop:reduction.groupoid}
	With the above conditions, $\gr{P_{\overline{\cov}}(S)}$ is isomorphic, as topological groupoids, to the reduction of $\grp$ to $\grp^{(0)}_{\cov}$.
\end{proposition}

\begin{proof}
	 The coverages $\cov'$ and $\overline{\cov}$ gives two nuclei, $\nu'$ and $\overline{\nu}$ respectively, on $C(S)$ as in Proposition \ref{prop:nucleus.from.cov}. Since for each $a\in S$, $\cov'\scj\overline{\cov}$, we can factor $\overline{\nu}$ through $\nu'$ and get a nucleus $\nu:P_{\cov'}(S)\to P_{\cov'}(S)$ in such way that $\nu(P_{\cov'}(S))=P_{\overline{\cov}}(S)$. From Theorem \ref{thm:nucleus.pseudogroup.hom.grp}, we have an embedding $\Phi:\gr{P_{\overline{\cov}(S)}}\to \gr{P_{\cov'}(S)}$ such that the image is the reduction of $\gr{P_{\cov'}(S)}$ to the image of $\gr{P_{\overline{\cov}(S)}}^{(0)}$.
	 
	 The nucleus $\nu$ restricts to a nucleus on $F_{\overline{\cov}_D}(D)$. Using the isomorphism from Lemma \ref{lemma:restiction.cov}(\ref{lemma:restiction.cov.unit.space}), we can restrict $\Phi$ to $\Sp{F_{\overline{\cov}_D}(D)}$, which by Lemma \ref{lemma:filter.doesnt.depend.on.a}\eqref{lemma:germ.equal.preimage} is taking preimages, which is $\Sp{\nu}$.
	 
	By Lemma \ref{lemma:semigroup.generates.pseudogroup}, $\bis(\grp)\cong P_{\cov'}(S)$. This gives an isomorphism between $\gr{\bis(\grp)}$ and $\gr{P_{C'}(S)}$. Since $\grp$ is sober the map $\eta$ given in Subsection \ref{subsec:groupoid} gives a topological groupoid isomorphism $\grp\cong \gr{\bis(\grp)}$. 
	
	We can then build the following commutative diagram
	\[
	\begin{tikzcd}
		\gr{P_{\overline{C}(S)}} \ar[r,hook] & \gr{P_{C'}(S)} \ar[r,"\simeq"] & \gr{\bis(\grp)} \ar[r,"\simeq"] & \grp \\
		\gr{P_{\overline{C}(S)}}^{(0)} \ar[r,hook] \ar[u,hook] & \gr{P_{C'}(S)}^{(0)} \ar[r,"\simeq"] \ar[u,hook] & \gr{\bis(\grp)}^{(0)} \ar[r,"\simeq"] \ar[u,hook] & \grp^{(0)} \ar[u,hook] \\
		\Sp{F_{\overline{\cov}_D}(D)} \ar[u,"\simeq"] \ar[r,hook] & \Sp{F_{\cov'_D}(D)} \ar[u,"\simeq"] \ar[urr, bend right = 10] & &.
	\end{tikzcd}
	\]
	
	By the discussion in the beginning of the proof and the discussion before the proposition, we see that, via the diagram, $\gr{P_{\overline{\cov}}(S)}$ is sent to the reduction of $\grp$ to $\grp^{(0)}_{\cov}$, from where the result follows.	
\end{proof}


\section{Tight coverages and tight filters}\label{sec:tight}

We will now use the theory developed in this paper to study tight filters and tight groupoids introduced by Exel in \cite{MR2419901}. First we deal with tight filters. For that, fix $E$ a semilattice with $0$ and denote by $\filt(E)$ the set of all filters in $E$. Using the characteristic function $\chi_\xi$ of a filter $\xi\in\filt(E)$, we can define an injective map $\xi\in\filt(E)\to\chi_\xi\in\{0,1\}^E$. Using the product topology on $\{0,1\}^E$, where $\{0,1\}$ has the discrete topology, we can induce a topology on $\filt(E)$. This coincides with the patch topology described in Subsection \ref{subsec:groupoid}, when we identify $\filt(E)$ with $\grl(E)$. If we denote by $\ufilt(E)$ the set of all ultrafilters, we can define a \emph{tight filter} as an element of the closure of $\ufilt(E)$ in $\filt(E)$ with the patch topology. The set of all tight filters will be denoted by $\tfilt(E)$. This is not Exel's original definition, but a consequence of \cite[Theorem 12.9]{MR2419901}, but for our purposes, this description suffices.

Another way of defining tight filters is using coverages as done in \cite{MR3077869} and based on \cite{MR2974110,MR2465914}. For each $a\in E$, we define $\covt(a)$ to be all finite subsets $Z\scj\lset{a}$ such that for every $0\neq b\leq a$, there exists $z\in Z$ such that $bz\neq 0$. Then $\covt=\{\covt\}_{a\in E}$ is a strong coverage on $E$. For each $a\in E$, the elements of $\covt(a)$ will be called tight coverings of $a$. By \cite[Proposition 2.25]{MR2974110}, a filter $\xi$ is tight if and only if for every $a\in\xi$ and every tight covering $Z$ of $a$, we have that $Z\cap\xi\neq\emptyset$.

One question that arises is: what is the relationship between $\tfilt(E)$ and the frame $F_\covt(E)$. In order to answer this question, we need to consider another topology on $\tfilt(E)$. For each $e\in E$, we define $V_e=\{\xi\in\tfilt(E)\mid e\in\xi\}$. The family $\{V_e\}_{e\in E}$ is a basis for a topology on $E$, which we will denote by $\tau_E$. In general, $\tau_E$ is coarser than $\tau_{patch}$. This other topology on $\tfilt(E)$ can also be induced from $\{0,1\}^E$, but we have to consider $\{0,1\}$ with the topology $\{\emptyset,\{1\},\{0,1\}\}$, ie the Sierpi\'nski space. Before we begin answering the posed question, we make a final observation: $V_e\neq\emptyset$ for all $e\in E\setminus\{0\}$, since the filter $\uset{e}$ is contained in some ultrafilter by an easy application of Zorn's Lemma.

\begin{lemma}\label{lemma:union.Ve.tight.filters}
	 If $e\in E$ and $Z$ is a tight cover of $e$, then $V_e=\bigcup_{z\in Z} V_z$.
\end{lemma}

\begin{proof}
	For any $z\in Z$, we have that $z\leq e$ so that $\bigcup_{z\in Z} V_z\scj V_e$. On the hand if $\xi\in V_e$, by the definition of tight filter, we have that $z\in\xi$ for some $z\in Z$. This implies the other inclusion $V_e\scj\bigcup_{z\in Z} V_z$.
\end{proof}

\begin{proposition}\label{prop:tight.filters.sober}
	The space $(\tfilt(E),\tau_E)$ is sober.
\end{proposition}

\begin{proof}
	We first verify that $(\tfilt(E),\tau_E)$ is $T_0$. Given $\xi,\eta\in\tfilt(E)$ such that $\xi\neq\eta$, suppose without loss of generality that there exists $e\in\xi\setminus\eta$, then $V_e$ is such that $\xi\in V_e$, but $\eta\notin V_e$.
	
	Now, suppose that $U$ is a meet-irreducible open proper subset of $\tfilt(E)$ and define the set $\xi=\{e\in E\mid V_e\nsubseteq U\}$. We prove that $\xi\in\tfilt(E)$ and that $U=\tfilt(E)\setminus\overline{\{\xi\}}$. By hypothesis, there exists $\eta\in\tfilt(E)\setminus U$. For any $e\in\eta$, we have that $\eta\in V_e$ so that $e\in U$, that is, $U\neq\emptyset$. We have that $0\notin \xi$ because $V_0=\emptyset\scj U$. If $e\in\xi$ and $f\in E$ is such that $e\leq f$, then $V_e\scj V_f$, which implies that $f\in\xi$. For $e,f\in\xi$, since $V_{ef}=V_e\cap V_f$ and $U$ is meet-irreducible, we have that $ef\in\xi$. By Lemma \ref{lemma:union.Ve.tight.filters}, if $e\in\xi$ and $Z$ is a tight covering of $e$, then there exists $z\in Z$ such that $z\in\xi$. We conclude that $\xi\in\tfilt(E)$.
	
	Observe that $\tfilt(E)\setminus\overline{\{\xi\}}=\bigcup_{e\in E\setminus\xi}V_e$. Since $U$ is the union of the basics open sets contained in $U$, by the definition of $\xi$, we have that $U=\bigcup_{e\in E\setminus\xi}V_e=\tfilt(E)\setminus\overline{\{\xi\}}$. The result follows.
\end{proof}

\begin{lemma}\label{lemma:tight.nucleus}
	Let $\nu:C(E)\to C(E)$ be the nucleus corresponding to the tight coverage $\covt$ as in Proposition \ref{prop:nucleus.from.cov}. Then for all $e\in E$, we have that $\nu(\lset{e})=\{g\in E\mid \exists Z\text{ tight covering of }g\text{ such that }Z\scj\lset{e}\}$.
\end{lemma}

\begin{proof}
	Fix $e\in E$ and define $I=\{g\in E\mid \exists Z\text{ tight covering of }g\text{ such that }Z\scj\lset{e}\}$. We prove that $I$ is the smallest $\covt$-ideal containing $e$, so that $I=\nu(\lset{e})$. Clearly $e\in I$ since $\{e\}$ is a tight covering of $e$.
	
	Take $g\in I$ and $f\leq g$. There exists $Z$ a tight covering of $g$ such that $Z\scj\lset{e}$. Then $fZ$ is a tight covering of $f$ such that $fZ\scj\lset{e}$, whence $f\in I$.
	
	Suppose now that $f\in E$ is such that there exists a tight covering $X=\{x_1,\ldots,x_n\}$ of $f$ such that $X\scj I$. By the definition of $I$, for each $i=1,\ldots,n$, there exists a tight covering $Z_i$ of $x_i$ such that $Z_i\scj\lset{e}$. Consider $Z=\bigcup_{i=1}^nZ_i$ and observe that $Z\scj\lset{e}$. We claim that $Z$ is a tight covering of $f$. Indeed, if $0\neq g\leq f$, then there exists $i\in\{1,\ldots,n\}$ such that $x_ig\neq 0$. Since $x_ig\leq x_i$, there exists $z\in Z_i$ such that $zg=zx_ig\neq 0$. Again using the definition of $I$, we conclude that $f\in I$.
	
	Now if $J$ is another $\covt$-ideal containing $e$ and $g\in I$, then there exists $Z$ tight covering of $g$ such that $Z\scj\lset{e}\scj J$, and therefore $g\in J$. Hence $I\scj J$.
\end{proof}

\begin{lemma}\label{lemma:filter.disjoint.from.ideal}
	Let $I$ be a $\covt$-ideal and suppose that $e\in E\setminus I$. Then, there exists $\xi\in \tfilt(E)$ such that $e\in \xi$ and $\xi\cap I=\emptyset$.
\end{lemma}

\begin{proof}
	Since $e\notin I$, the filter $\uset{e}$ is a filter such that $\uset{e}\cap I=\emptyset$ and $e\in \uset{e}$. Using Zorn's lemma, we can find a maximal filter $\xi$ with the property that $e\in\xi$ and $\xi\cap I=\emptyset$. Suppose that $\xi$ is not tight. Then, there exist $f\in\xi$ and $Z=\{z_1,\ldots,z_n\}$ a tight covering of $f$ such that $Z\cap\xi=\emptyset$.
	
	We claim that for every $z\in Z$, there exists $g\in \xi$ such that $gz\in I$. If $z$ is such that $gz=0$, for some $g\in\xi$, then we can choose this $g$, since $0\in I$. For $z$ such that $gz\neq 0$ for all $g\in\xi$, then $\uset{z\xi}$ if a filter strictly larger then $\xi$. By the maximality of $\xi$, we have that $I\cap \uset{z\xi}\neq \emptyset$, from where it follows the existence of $g\in\xi$ with $gz\in I$.
	
	For each $i=1,\ldots,n$, let $g_i\in\xi$ be such that $g_iz_i\in I$, and define $h=fg_1\ldots g_n$. We notice that $h\in\xi$, and $hZ$ is a tight covering of $h$. Since for each $i=1,\ldots,n$, we have that $hz_i\leq g_iz_i$, which implies that $hZ\scj I$. Since $I$ is a $\covt$-ideal, this implies that $h\in I$, contradicting the fact that $\xi\cap I=\emptyset$.
\end{proof}

\begin{theorem}\label{thm:tight.cov.weak.top}
	Let $E$ be an semilattice with 0 and $\covt$ the tight coverage on $E$. Then $F_{\covt}(E)\cong \tau_E$ as frames.
\end{theorem}

\begin{proof}
	By Lemma \ref{lemma:union.Ve.tight.filters}, the map $e\in E\mapsto V_e\in\tau_E$ is a $\covt$-cover-to-join map. By the universal property of $F_{\covt}(E)$ there is frame homomorphism $\phi:F_{\covt}(E)\to\tau_E$ such that $\phi(\pi(e))=V_e$, where $\pi(e)=\nu(\lset{e})$ as in Theorem \ref{thm:univesal.property}.
		
	Since $B=\{V_e\}_{e\in E}$ is a subsemilattice of $\tau_E$ that is a basis, by Lemma \ref{lemma:semigroup.generates.pseudogroup}, the coverage $\cov$ on $B$ given by the union relations on $B$ is such that $\tau_E\cong F_{\cov}(B)$. We want to define a map $\theta:B\to F_{\covt}(E)$ by $\theta(V_e)=\pi(e)$. In order to show that $\theta$ is well-defined, we have to show that for $e,f\in E$ such that $V_e=V_f$, we have that $\pi(e)=\pi(f)$. Suppose that $g\in\pi(e)$, then by Lemma \ref{lemma:tight.nucleus}, there exists a tight covering $Z$ of $g$ such that $Z\scj\lset{e}$. We claim that the set $fZ=\{fz\mid z\in Z\}$ is also a tight covering of $g$. Let $h\leq g$ be such that $h\neq 0$, then there exists $z_0$ such that $z_0h\neq 0$. Since $V_{z_0h}\neq\emptyset$, there exists $\xi\in\tfilt(E)$ such that $z_0h\in\xi$. Observe that $z_0h\leq z_0\leq e$, so that $\xi\in V_e=V_f$. This implies that $f\in\xi$ and, since $\xi$ is a filter $fz_0h\neq 0$. This proves that $fZ$ is a tight covering of $g$, which, by Lemma \ref{lemma:tight.nucleus} implies that $g\in \pi(f)$. Hence $\pi(e)\scj \pi(f)$, and the other inclusion is analogous.
	
	For $e,f\in E$, we have that $\theta(V_e\cap V_f)=\theta(V_{ef})=\pi(ef)=\pi(e)\pi(f)$, and hence $\theta$ is a semilattice homomorphism. In order to use the universal property of $F_{\cov}(B)$, we have to prove that $\theta$ is a $\cov$-cover-to-join map. Let $e\in E$ and $X\scj E$ be such that $V_e=\bigcup_{x\in X} V_x$. Since $\theta$ preserves the order, $\bigvee_{x\in X} \pi(x) \leq \pi(e)$. By Lemma \ref{lemma:filter.disjoint.from.ideal}, if $e\notin \bigvee_{x\in X} \pi(x)$, there would exist a filter $\xi\in V_e$ such that $\xi\cap \bigvee_{x\in X} \pi(x)=\emptyset$, which would imply that $\xi\notin \bigcup_{x\in X} V_x$. Hence $e\in \bigvee_{x\in X} \pi(x)$, and since $\pi(e)$ is the smallest $\covt$-ideal containing $e$, we get $\pi(e)\leq \bigvee_{x\in X} \pi(x)$. We conclude that $\theta$ is a $\cov$-cover-to-join map. Using the universal property of $F_{\cov}(B)$ and the isomorphism $F_{\cov}(B)\cong\tau_E$, we find a frame homomorphism $\psi:\tau_E\to F_{\covt}(E)$ such that $\psi(V_e)=\pi(e)$. Then $\psi$ is the inverse of $\phi$ and $\phi$ is a frame isomorphism.
\end{proof}

\begin{corollary}
	The space $(\tfilt(E),\tau_E)$ is homeomorphic to $\Sp{F_{\covt}(E)}$.
\end{corollary}

\begin{proof}
	By Proposition \ref{prop:tight.filters.sober} $(\tfilt(E),\tau_E)$ is sober, so that $(\tfilt(E),\tau_E)$ is homeomorphic to $\Sp{\tau_E}$. From Theorem \ref{thm:tight.cov.weak.top}, we get that $(\tfilt(E),\tau_E)$ is homeomorphic to $\Sp{F_{\covt}(E)}$.
\end{proof}

Now, we work with Exel's tight groupoid defined in \cite{MR2419901}. For that, let $S$ be an inverse semigroup with $0$. Exel's definition uses the natural action of $S$ on $\tfilt(E(S))$, however we will use Lawson and Lenz's description. They showed that Exel's tight groupoid is the reduction of Paterson's universal groupoid $\grl(S)$ to $\tfilt(E(S))$ with the patch topology \cite[Subsection 5.2]{MR3077869}. We want to arrive at this same result using the results proved in this paper. For that we extend the definition of the tight coverage to $S$ as in \cite{MR3077869}. For each $s\in S$, we define we define $\covt(s)$ to be all finite subsets $Z\scj\lset{s}$ such that for every $0\neq t\leq s$, there exists $z\in Z$ such that $\lset{y}\cap\lset{z}\neq \{0\}$. We can then use the tight coverage and the map $\theta:S\to\bis(\grp_u(S))$ given by $\theta(s)=U_s$ together with Lemma \ref{lem:from.cov.inv.semigroup.to.cov.pseudogroup} to find a coverage on $\bis(\grp_u(S))$. By means of Propositions \ref{prop:restriction.cov.idempotents} and \ref{prop:reduction.groupoid}, it suffices to work with the idempotents of $S$.

\begin{theorem}\label{thm:reduction.filters.to.tight}
	Let $E$ be a semilattice with $0$, $\filt(E)$ the set of all filter in $E$ and $\tau_{patch}$ the patch topology on $E$. Consider $\theta:E\to \tau_{patch}$ given by $\theta(e)=U_e$ and $\cov$ the coverage on $\tau_{patch}$ induced by $\theta$ and the tight coverage $\covt$ on $E$. Then $\tfilt(E)=\filt(E)_{\cov}$.
\end{theorem}

\begin{proof}
	Since $(\filt(E),\tau_{patch})$ is Hausdorff, we can use Lemma \ref{lem:subspace.from.coverage}. Let $\xi\in\tfilt(E)$ and $U\in\tau_{patch}$ such that $\xi\notin U$. For a covering $X$ of $U$ such that $\bigcup X=U$, it is clear that $\xi\notin W$ for all $W\in X$. For the other kind of coverings, we have that $U=V_e\cap W$, where $e\in E$ and $W\in\tau_{patch}$. Let $Z$ be a tight covering of $e$, so that $\{V_z\cap W\}_{z\in Z}$ is a covering of $U$. Since $\xi\notin U$, either $e\notin \xi$ or $\xi\notin W$. In the first case, since $\xi$ is a tight filter, we have that $z\notin\xi$ for all $z\in Z$ so that $\xi\notin V_z\cap W$. In the second case, it immediate that $\xi\notin V_z\cap W$ for all $z\in Z$. Hence $\xi\in \filt(E)_{\cov}$.
	
	Suppose now that $\xi\in\filt(E)_{\cov}$. For $e\in\xi$ and $Z$ tight covering of $e$. In this case $\{V_z\}_{z\in Z}$ is a covering of $V_e$. By Lemma \ref{lem:subspace.from.coverage}, $\xi\in V_z$ for some $z\in Z$, that is, $z\in\xi$ for some $z\in Z$. Hence $\xi\in\tfilt(E)$.
\end{proof}

\begin{corollary}
	Let $S$ be an inverse semigroup with 0, $\grp_u(S)$ the groupoid of all filters of $S$ with the patch topology $\tau_{patch}$. Then the tight groupoid is the reduction of $\grp_u(S)$ to the tight filters on $E(S)$.
\end{corollary}

\begin{proof}
	Consider $\theta:S\to\bis(\grp_u(S))$ given by $\theta(s)=U_s$ and $\cov$ the coverage on $\tau_{patch}$ induced by $\theta$ and the tight coverage $\covt$ on $S$. By Proposition \ref{prop:restriction.cov.idempotents} we can restrict the coverage $\cov$ the idempotents getting the coverage of Theorem \ref{thm:reduction.filters.to.tight}. The result then follows from Theorem \ref{thm:reduction.filters.to.tight} and Proposition \ref{prop:reduction.groupoid}.
\end{proof}

%
%
%

\bibliographystyle{abbrv}
\bibliography{cov_ref}

\begin{thebibliography}{10}

\bibitem{MR3648984}
G.~Boava, G.~G. de~Castro, and F.~de~L.~Mortari.
\newblock Inverse semigroups associated with labelled spaces and their tight
  spectra.
\newblock {\em Semigroup Forum}, 94(3):582--609, 2017.

\bibitem{MR3680957}
G.~Boava, G.~G. de~Castro, and F.~de~L.~Mortari.
\newblock {${\rm C}^*$}-algebras of labelled spaces and their diagonal {${\rm
  C}^*$}-subalgebras.
\newblock {\em J. Math. Anal. Appl.}, 456(1):69--98, 2017.

\bibitem{Gil3}
G.~Boava, G.~G. de~Castro, and F.~d.~L. Mortari.
\newblock Groupoid models for the {C}*-algebra of labelled spaces.
\newblock {\em Bulletin of the Brazilian Mathematical Society, New Series},
  2019.

\bibitem{MR1303779}
A.~Connes.
\newblock {\em Noncommutative geometry}.
\newblock Academic Press, Inc., San Diego, CA, 1994.

\bibitem{de_castro_2020}
G.~G. de~Castro.
\newblock Recovering the boundary path space of a topological graph using
  pointless topology.
\newblock {\em Journal of the Australian Mathematical Society}, page 1–17,
  2020.

\bibitem{MR2419901}
R.~Exel.
\newblock Inverse semigroups and combinatorial {$C\sp \ast$}-algebras.
\newblock {\em Bull. Braz. Math. Soc. (N.S.)}, 39(2):191--313, 2008.

\bibitem{MR0009426}
I.~Gelfand and M.~Naimark.
\newblock On the imbedding of normed rings into the ring of operators in
  {H}ilbert space.
\newblock {\em Rec. Math. [Mat. Sbornik] N.S.}, 12(54):197--213, 1943.

\bibitem{MR698074}
P.~T. Johnstone.
\newblock {\em Stone spaces}, volume~3 of {\em Cambridge Studies in Advanced
  Mathematics}.
\newblock Cambridge University Press, Cambridge, 1982.

\bibitem{MR2067120}
T.~Katsura.
\newblock A class of {$C^\ast$}-algebras generalizing both graph algebras and
  homeomorphism {$C^\ast$}-algebras. {I}. {F}undamental results.
\newblock {\em Trans. Amer. Math. Soc.}, 356(11):4287--4322, 2004.

\bibitem{MR1694900}
M.~V. Lawson.
\newblock {\em Inverse semigroups}.
\newblock World Scientific Publishing Co., Inc., River Edge, NJ, 1998.
\newblock The theory of partial symmetries.

\bibitem{MR2974110}
M.~V. Lawson.
\newblock Non-commutative {S}tone duality: inverse semigroups, topological
  groupoids and {$C^\ast$}-algebras.
\newblock {\em Internat. J. Algebra Comput.}, 22(6):1250058, 47, 2012.

\bibitem{MR3077869}
M.~V. Lawson and D.~H. Lenz.
\newblock Pseudogroups and their \'etale groupoids.
\newblock {\em Adv. Math.}, 244:117--170, 2013.

\bibitem{MR3109745}
M.~V. Lawson, S.~W. Margolis, and B.~Steinberg.
\newblock The \'{e}tale groupoid of an inverse semigroup as a groupoid of
  filters.
\newblock {\em J. Aust. Math. Soc.}, 94(2):234--256, 2013.

\bibitem{MR2465914}
D.~H. Lenz.
\newblock On an order-based construction of a topological groupoid from an
  inverse semigroup.
\newblock {\em Proc. Edinb. Math. Soc. (2)}, 51(2):387--406, 2008.

\bibitem{MR1724106}
A.~L.~T. Paterson.
\newblock {\em Groupoids, inverse semigroups, and their operator algebras},
  volume 170 of {\em Progress in Mathematics}.
\newblock Birkh\"auser Boston, Inc., Boston, MA, 1999.

\bibitem{MR1962477}
A.~L.~T. Paterson.
\newblock Graph inverse semigroups, groupoids and their {$C^\ast$}-algebras.
\newblock {\em J. Operator Theory}, 48(3, suppl.):645--662, 2002.

\bibitem{MR2868166}
J.~Picado and A.~Pultr.
\newblock {\em Frames and locales}.
\newblock Frontiers in Mathematics. Birkh\"auser/Springer Basel AG, Basel,
  2012.
\newblock Topology without points.

\bibitem{MR584266}
J.~Renault.
\newblock {\em A groupoid approach to {$C^{\ast} $}-algebras}, volume 793 of
  {\em Lecture Notes in Mathematics}.
\newblock Springer, Berlin, 1980.

\bibitem{resende2006lectures}
P.~Resende.
\newblock Lectures on {\'e}tale groupoids, inverse semigroups and quantales.
\newblock In {\em Lecture notes for the GAMAP IP Meeting, Antwerp}, page 115.
  Citeseer, 2006.

\bibitem{MR2304314}
P.~Resende.
\newblock \'{E}tale groupoids and their quantales.
\newblock {\em Adv. Math.}, 208(1):147--209, 2007.

\bibitem{MR0325820}
B.~M. Schein.
\newblock Completions, translational hulls and ideal extensions of inverse
  semigroups.
\newblock {\em Czechoslovak Math. J.}, 23(98):575--610, 1973.

\bibitem{MR1507106}
M.~H. Stone.
\newblock Postulates for {B}oolean {A}lgebras and {G}eneralized {B}oolean
  {A}lgebras.
\newblock {\em Amer. J. Math.}, 57(4):703--732, 1935.

\bibitem{MR1002193}
S.~Vickers.
\newblock {\em Topology via logic}, volume~5 of {\em Cambridge Tracts in
  Theoretical Computer Science}.
\newblock Cambridge University Press, Cambridge, 1989.

\bibitem{MR2301938}
T.~Yeend.
\newblock Groupoid models for the {$C^*$}-algebras of topological higher-rank
  graphs.
\newblock {\em J. Operator Theory}, 57(1):95--120, 2007.

\end{thebibliography}
	
\end{document}